\documentclass[11pt, reqno]{amsart} 
\usepackage{amssymb, latexsym, amsmath}
\newtheorem{theorem}{Theorem}[section]
\newtheorem{lemma}[theorem]{Lemma}

\newtheorem{corollary}[theorem]{Corollary}
\theoremstyle{definition}
\newtheorem{remark}[theorem]{Remark} 
\theoremstyle{definition}
\newtheorem{definition}[theorem]{Definition} 
\theoremstyle{definition}

\theoremstyle{definition}

\numberwithin{equation}{section}

\usepackage[mathscr]{eucal}

\usepackage{xy} 
\xyoption{all}

\newcommand{\mf}[1]{\mathfrak{#1}}

\newcommand{\del}{\partial}
\newcommand{\Oh}{\mathcal{O}}

\DeclareMathOperator{\Spec}{Spec}

\DeclareMathOperator{\Ass}{Ass}
\DeclareMathOperator{\Ann}{Ann}
\DeclareMathOperator{\nil}{nil}
\DeclareMathOperator{\Der}{Der}

\CompileMatrices

\begin{document}
\title[Principal Gradient Schemes]{Principal gradient schemes have
regular reduced closed subschemes}
\date{\today}

\author[J. Mullet]{Joshua P. Mullet}
\address{Department of Mathematics\\
        The Ohio State University\\
        Columbus, OH 43210}
\email{mullet@math.ohio-state.edu}
\maketitle

\section{Introduction}\label{S:intro}

This paper represents the first steps in a program whose goal is to
understand the formal properties of \emph{gradient schemes}, i.e.,
schemes that are locally analytically cut out by the gradient of a
function (see Definition~\ref{D:gradscheme}).  In \cite{hC05}, Clemens
has shown that Hilbert Schemes of curves on $K$-trivial threefolds are
gradient schemes.  Therefore it is the hope that an understanding of
gradient schemes will shed light on the geometry of these Hilbert
schemes.  The problem of understanding gradient schemes is also
interesting from the point of view of commutative algebra in that we are
trying to determine which ideals in power series rings are gradient
ideals.  

The contents of the paper are as follows.  In Section~\ref{S:gradjac} we
review some basic algebraic facts and definitions about Jacobian ideals
and gradient ideals in power series rings.  In Section~\ref{S:regcrit}
we prove a regularity criterion (Theorem~\ref{T:reduced}) for the
reduced quotient ring of a power series ring modulo a principal ideal.
This criterion is stated in terms of the associated primes of the
Jacobian ideal and is interesting in its own right.  We are currently
working on generalizations.  Finally, in Section~\ref{S:gradapp} we
prove our main result (Theorem~\ref{T:princgrad}), which states that
principal gradient schemes have regular reduced subschemes.  We consider
complete intersection gradient schemes in \cite{jMUP}.

\subsection{Acknowledgments}\label{Ss:thanks}

I am indebted to Herb Clemens for his guidance and for
suggesting this project.  I would also like to thank Linda Chen
and Sheldon Katz for valuable conversations.  Finally, I would to thank
Yu-Han Liu who has shown me an elementary proof of
Theorem~\ref{T:princgrad} that does not use Theorem~\ref{T:reduced}.

\section{Basic facts regarding gradient and Jacobian
ideals}\label{S:gradjac}

Let $k$ be a field of characteristic zero and consider the ring of
formal power series $k[[x_1, \ldots, x_n]]$.  We gather some basic
facts and definitions regarding Jacobian and gradient ideals in $k[[x_1,
\ldots, x_n]]$.

\begin{definition}\label{D:gradjacdef}
Let $f \in k[[x_1, \ldots, x_n]]$ be a power series.  The \emph{gradient
ideal} of $f$ is the ideal
\[
    I_{\nabla f} := \left (\dfrac{\del f}{\del x_1}, \ldots, \dfrac{\del
    f}{\del x_n} \right) \subseteq k[[x_1, \ldots, x_n]],
\]
and the \emph{Jacobian ideal} of $f$ is the ideal
\[
    (f, I_{\nabla f}) \subset k[[x_1, \ldots, x_n]].
\]
\end{definition}

\begin{remark}\label{R:partials}
We may also write
\[
    \dfrac{\del f}{\del x_i}
\]
as $f_{x_i}$.
\end{remark}

\begin{definition}\label{D:order}
Let $f \in k[[x_1, \ldots, x_n]]$ be a nonzero power series.  The
\emph{order} of $f$ is the minimal total degree of all nonzero monomials
appearing in $f$.
\end{definition}

\begin{lemma}\label{L:isolated}
If $f \in k[[x_1, \ldots, x_n]]$ is a squarefree power series, then the
ring 
\[
    \dfrac{k[[x_1, \ldots, x_n]]}{(f, I_{\nabla f})}
\]
has dimension less than $n-1$.
\end{lemma}

\begin{proof}
If 
\[
    \dim \dfrac{k[[x_1, \ldots, x_n]]}{(f, I_{\nabla f})} = n - 1,
\]
then there exists an irreducible factor $h$ of $f$ such that for all
$1 \leq i \leq n$ the power series $h$ divides the partial derivative of
$f$ with respect to $x_i$.  Write $f = h^kg$ where $h$ does not divide
$g$. Then for all $1 \leq i \leq n$ the factor $h$ divides 
\[
    \dfrac{\del f}{\del x_i} = h^k \dfrac{\del g}{\del x_i} + k h^{k-1}g
    \dfrac{\del h}{\del x_i}.
\]
It follows that $h$ divides all its partial derivatives, but this is
impossible since there must exist an $1 \leq i \leq n$ such that the
order of the series
\[
    \dfrac{\del h}{\del x_i}
\]
is less than the order of $h$.  
\end{proof}

\begin{lemma}\label{L:divpartials}
Let $h$ be an irreducible factor of a power series $f \in k[[x_1,
\ldots, x_n]]$.  If for all $1 \leq i \leq n$ we have 
\[
    h^k \textup{ divides } \dfrac{\del f}{\del x_i},
\]
then $h^{k+1}$ divides $f$.
\end{lemma}

\begin{proof}
Write $f = h^l g$, where $h$ does not divide $g$.  We wish to show that
$l \geq k+1$.  For $1 \leq i \leq n$ we compute
\begin{equation}\label{Eq:divpartials}
    \dfrac{\del f}{\del x_i} = h^{l-1}\left(h\dfrac{\del g}{\del x_i} +
    l\dfrac{\del h}{\del x_i} g\right).
\end{equation}
If $l < k+1$ then since $h^k$ divides the right hand side in
\eqref{Eq:divpartials}, it must be the case that $h$ divides 
\[
    \left(h\dfrac{\del g}{\del x_i} + l\dfrac{\del h}{\del x_i}
    g\right).
\]
Since $h$ does not divide $g$, it follows that $h$ must divide all of
its partials.  This contradicts Lemma~\ref{L:isolated}.
\end{proof}

\begin{lemma}\label{L:radgrad}
Let $f \in k[[x_1, \ldots, x_n]]$ be a power series that is not a unit.
Then 
\[
    f \in \sqrt{I_{\nabla f}}.
\]
\end{lemma}

\begin{proof}
This follows from the stronger statement that $f$ is in the integral
closure of the ideal $(x_1, \ldots, x_k) I_{\nabla f}$.  See
\cite{cH06}.
\end{proof}

\section{A regularity criterion}\label{S:regcrit}

We recall the Jacobian criterion for power series rings due to Nagata,
see \cite[Proposition~22.7.2]{EGAIVi}.  For any $k$-algebra $A$ we
denote by $\Der_k (A)$ the $A$-module of all $k$-linear derivations from
$A$ to $A$.

\begin{theorem}[Nagata]\label{T:jaccrit}
Let $I$ be an ideal of the ring $k[[x_1, \ldots, x_n]]$, and let $\mf{p}$ be a
prime ideal with $I \subseteq \mf{p}$.  Then the ring
\[
    \dfrac{k[[x_1, \ldots, x_n]]_\mf{p}}{I\cdot k[[x_1, \ldots, x_n]]_\mf{p}}
\]
is regular if and only if there exist elements $g_1, \ldots, g_k \in
I$ and derivations $D_1, \ldots, D_k \in \Der_k(k[[x_1, \ldots, x_n]])$ such that 
\begin{enumerate}
    \item the images of $g_1, \ldots, g_k$ in $I\cdot k[[x_1, \ldots,
    x_n]]_\mf{p}$ generate $I\cdot k[[x_1, \ldots, x_n]]_\mf{p}$, and
    \item $\det \{D_i g_j\} \notin \mf{p}$.
\end{enumerate}

\end{theorem}

\begin{corollary}\label{C:jaccrit}
Let $f \in k[[x_1, \ldots, x_n]]$ be an element that is not a unit.
Then the local ring 
\[      
    \dfrac{k[[x_1, \ldots, x_n]]}{(f)}
\]
is regular if and only if the ideal 
\begin{equation}\label{Eq:jaccritid}
    \left(f, I_{\nabla f}\right)
\end{equation}
is the unit ideal. 
\end{corollary}

\begin{proof}
First, suppose that the ideal \eqref{Eq:jaccritid} is the unit ideal.
Since $f$ is not a unit, this means that one of the partial
derivatives of $f$ is a unit.  We may assume that 
\[
    \dfrac{\del f}{\del x_1} \in k[[x_1, \ldots, x_n]]^\times.
\]
But then the derivation $\frac{\del}{\del x_1}$ and the element $f$ satisfy
the conditions of Theorem~\ref{T:jaccrit}, and 
\[
    \dfrac{k[[x_1, \ldots, x_n]]}{(f)}
\]
is regular. 

Now suppose that the ring 
\[
    \dfrac{k[[x_1, \ldots, x_n]]}{(f)}
\]
is regular.  Let $g_1, \ldots, g_k \in k[[x_1, \ldots, x_n]]$ and $D_1, \ldots, D_k
\in \Der_k(k[[x_1, \ldots, x_n]])$ be elements satisfying the conditions of
Theorem~\ref{T:jaccrit}.  Since $\det \{D_ig_j\}$ is a unit, it follows
that $D_lg_m$ must be a unit for some $l$ and $m$.  Recall
(\cite[Theorem~30.6]{hM89}) that $\Der_k(k[[x_1, \ldots, x_n]])$ is a free
$k[[x_1, \ldots, x_n]]$-module of rank two with basis
\[
    \left\{\dfrac{\del}{\del x_1}, \ldots, \dfrac{\del}{\del x_n}\right\}.
\]
Suppose 
\begin{equation}\label{Eq:derivation}
    D_l = \sum_{i =1}^{n} \alpha_i \dfrac{\del}{\del x_i} 
\end{equation}
for power series $\alpha_1, \ldots, \alpha_n \in k[[x_1,\ldots,x_n]]$.
Since $D_lg_m$ is a unit, one of the summands in \eqref{Eq:derivation}
applied to $g_m$ is a unit.  We may assume that 
\[
    \dfrac{\del g_m}{\del x_1} \in k[[x_1, \ldots, x_n]]^\times.
\]
But $f \mid g_m$ so we may write $g_m = fh$ for some $h \in k[[x_1, \ldots, x_n]]$.
Applying the product rule for partial differentiation, we find
\[
    \dfrac{\del g_m}{\del x_1} = f\dfrac{\del h}{\del x_1} + h\dfrac{\del
    f}{\del x_1}.
\]
Now $f$ is not a unit by assumption so it follows that 
\[
    \dfrac{\del f}{\del x_1} \in k[[x_1, \ldots, x_n]]^\times
\]
as required.
\end{proof}

We next establish a Jacobian criterion that can determine when the
reduced quotient ring of a ring is regular.  Recall that if $M$ is an
$R$-module, we say that a prime ideal $\mf{p} \subseteq R$ is an
\emph{associated prime} of $M$ if $\mf{p} = \Ann (R m)$ for some $m \in
M$.  The set of all associated primes of an $R$-module $M$ is called
$\Ass_R(M)$.  Recall that minimal elements of $\Ass_R(0)$ correspond to
irreducible components of $\Spec R$ and non-minimal elements correspond
to \emph{embedded components}.  Note that $\mf{p} \in \Ass_R (M)$ if and
only if there exists an injective $R$-linear map
\begin{equation}\label{Eq:assmap}
\xymatrix{
        {\dfrac{R}{\mf{p}}} \ar[r] & {M}.
}
\end{equation}

\begin{theorem}\label{T:reduced}
Let $f$ be an element of the ring $k[[x_1, \ldots, x_n]]$ that is not a
unit.  Then the local ring 
\[
    \frac{k[[x_1, \ldots, x_n]]}{\sqrt{(f)}}
\]
is regular if and only if all elements of the set
\[      
    \Ass_{k[[x_1, \ldots, x_n]]} \left( \frac{k[[x_1, \ldots, x_n]]}{(f,
    I_{\nabla f})} \right), 
\]
have height one.
\end{theorem}

\begin{proof}
First suppose that $\frac{k[[x_1, \ldots, x_n]]}{\sqrt{(f)}}$ is
regular.  It follows that $f = g^k$ for some irreducible element $g \in
k[[x_1, \ldots, x_n]]$.  For if $f$ were divisible by two distinct
irreducible factors $g$ and $h$ not unit multiples of each other, then
the quotient ring $\frac{k[[x_1, \ldots, x_n]]}{\sqrt{(f)}}$ would have
the cosets of $g$ and $h$ as zero-divisors.  But this would contradict
the fact that regular local rings are integral domains
\cite[Theorem~14.3]{hM89}.  Since $\sqrt{(f)} = (g)$, we find that the
ideal $(g, I_{\nabla g})$ is the unit ideal by Corollary~\ref{C:jaccrit}.
Hence,
\[
\begin{aligned}
    \dfrac{k[[x_1, \ldots, x_n]]}{(f, I_{\nabla f})} &= \dfrac{k[[x_1,
    \ldots, x_n]]}{(g^k,kg^{k-1}g_{x_1}, \ldots, kg^{k-1}g_{x_n})} \\
        &= \dfrac{k[[x_1, \ldots, x_n]]}{(g^{k-1}(g, I_{\nabla g}))} \\
        &= \dfrac{k[[x_1, \ldots, x_n]]}{(g^{k-1})},
\end{aligned}
\]
and
\[
    \Ass_{k[[x_1, \ldots, x_n]]}\left( \dfrac{k[[x_1, \ldots,
    x_n]]}{(g^{k-1})} \right) = 
    \begin{cases}
    \emptyset &\text{if } k = 1 \\
    (g) &\text{if } k > 1
    \end{cases}
\]
as required.

Now, suppose that the ring 
\[
    \dfrac{k[[x_1, \ldots, x_n]]}{\sqrt{(f)}}
\]
is not regular.  We may assume that 
\[
    f = \prod_{i = 1}^k g_i^{e_i} 
\]
for some positive integers $e_1, \ldots, e_k$ and where the elements
$g_1, \ldots, g_k$ are pairwise relatively prime irreducible factors of
$f$.  Then the radical of the ideal $(f)$ is given by 
\[      
    \sqrt{(f)} = \left( \prod_{i = 1}^k g_i \right).
\]
We then may compute the Jacobian ideal of $f$:
\begin{equation*}
\begin{aligned}
    (f, I_{\nabla f}) &= \left( \prod_{i=1}^k g_i^{e_i}, 
                    \prod_{i=1}^k g_i^{e_i}\left(\sum_{i=1}^k
                    \dfrac{e^i{g_i}_{x_1}}{g_i}\right), \ldots, 
                    \prod_{i=1}^k g_i^{e_i}\left(\sum_{i=1}^k
                    \dfrac{e^i{g_i}_{x_n}}{g_i}\right) \right) \\
                &= \left(\prod_{i=1}^k g_i^{e_i-1}\right)
                    \left( \prod_{i=1}^k g_i, 
                    \prod_{i=1}^k g_i\left(\sum_{i=1}^k
                    \dfrac{e^i{g_i}_{x_1}}{g_i}\right), \ldots, 
                    \prod_{i=1}^k g_i\left(\sum_{i=1}^k
                    \dfrac{e^i{g_i}_{x_n}}{g_i}\right) \right). 
\end{aligned}
\end{equation*}
Since ${g_i}$ cannot divide all of its partial derivatives by
Lemma~\ref{L:isolated}, it follows that the ideal
\begin{equation}\label{Eq:jacred}      
     \left( \prod_{i=1}^k g_i, 
                    \prod_{i=1}^k g_i\left(\sum_{i=1}^k
                    \dfrac{e^i{g_i}_{x_1}}{g_i}\right), \ldots
                    \prod_{i=1}^k g_i\left(\sum_{i=1}^k
                    \dfrac{e^i{g_i}_{x_n}}{g_i}\right) \right)
\end{equation}
is either the unit ideal or has height greater than two.  An ideal
\[
    (\phi_1, \ldots, \phi_s) \subseteq k[[x_1, \ldots, x_n]]
\]
is the unit ideal if and only if $\phi_i$ is a unit for some $1 \leq i
\leq s$, and a power series $\psi \in k[[x_1, \ldots, x_n]]$ is a unit
if and only if its constant term is not zero.  For $1 \leq j \leq n$,
the constant term of 
\[
                    \prod_{i=1}^k g_i\left(\sum_{i=1}^k
                    \dfrac{e^i{g_i}_{x_j}}{g_i}\right)
\]
is the sum of the constant terms of the elements
\begin{equation}\label{Eq:derivpieces}
    e^ig_i \cdots g_{i-1} {g_i}_{x_j} g_{i+1} \cdots g_k
\end{equation}
for $1 \leq i \leq k$.  If $k \geq 2$, the constant term in
$\eqref{Eq:derivpieces}$ is zero and the
ideal in \eqref{Eq:jacred} is not the unit ideal. Hence, the ideal
\eqref{Eq:jacred} has height two if $k \geq 2$, and the theorem follows.
If $k = 1$, we put $g_1 = g$ and $e_1 = e$ and compute
\[
\begin{aligned}
    (f, I_{\nabla f}) &= (g^e, eg^{e-1}g_{x_1},\ldots,  eg^{e-1}g_{x_n}) \\
                &= (g^{e-1})(g, I_{\nabla g}),
\end{aligned}
\]
and $(g, I_{\nabla g})$ is not the unit ideal by Corollary~\ref{C:jaccrit}
because we are assuming that the ring
\[
    \dfrac{k[[x_1, \ldots, x_n]]}{\sqrt{(f)}} = \dfrac{k[[x_1, \ldots,
    x_n]]}{(g)} 
\]
is not regular.  Since $g$ is irreducible, Lemma~\ref{L:isolated}
implies that 
\[
    \dim \dfrac{k[[x_1, \ldots, x_n]]}{(g, I_{\nabla g})} < n-1.
\] 
Hence, there exists a prime ideal
\[
    \mf{p} \in \Ass_{k[[x_1, \ldots, x_n]]} 
    \left( \dfrac{k[[x_1, \ldots, x_n]]}{(g, I_{\nabla g})} \right)
\]
such that $\mf{p}$ has height two or more. 

We are now in the situation where 
\[
    \dfrac{k[[x_1, \ldots, x_n]]}{(f, I_{\nabla f})} = \dfrac{k[[x_1,
    \ldots, x_n]]}{(\psi)\cdot I} 
\]
for some element $\psi \in k[[x_1, \ldots, x_n]]$ and some ideal $I$
such that  
\[
    \mf{p} \in \Ass_{k[[x_1, \ldots, x_n]]} \left( \dfrac{k[[x_1,
    \ldots, x_n]]}{I} \right)   
\]
As in \eqref{Eq:assmap}, there is an injective $k[[x_1, \ldots,
x_n]]$-linear map 
\[
\xymatrix{
    {\dfrac{k[[x_1, \ldots, x_n]]}{\mf{p}}} \ar[r] & {\dfrac{k[[x_1,
    \ldots, x_n]]}{I}}.  
}
\]
Composing with the injective $k[[x_1, \ldots, x_n]]$-linear map
\[
\xymatrix{
    {\dfrac{k[[x_1, \ldots, x_n]]}{I}} \ar[r] & {\dfrac{k[[x_1, \ldots,
    x_n]]}{(\psi)\cdot I}} }
\]
induced by multiplication by the element $\psi$ shows that 
\[
    \mf{p} \in \Ass_{k[[x_1, \ldots, x_n]]}\left(
    \dfrac{k[[x_1, \ldots, x_n]]}{(\psi)\cdot I} \right)
\] 
as required.
\end{proof}

\section{Application to gradient ideals and gradient
schemes}\label{S:gradapp}

\begin{definition}\label{D:gradideal}
Let $I$ be an ideal of the power series ring $k[[x_1, \ldots, x_n]]$.
We say that $I$ is a \emph{gradient ideal} if there exists an element $f
\in k[[x_1, \ldots, x_n]]$ such that $I = I_{\nabla f}$. 
\end{definition}

\begin{lemma}\label{L:gradinvar}
The property of an ideal $I \subseteq k[[x_1, \ldots, x_n]]$ being a
gradient ideal is invariant under isomorphism, the gradient ideal of an
element being sent to the gradient ideal of the image of the element
under the isomorphism.
\end{lemma}

\begin{proof}
Let $\theta$ be an isomorphism
\begin{equation}        
\xymatrix{
    {k[[x_1, \ldots, x_n]]} \ar[r]^-*{\theta} & {k[[u_1, \ldots, u_n]],}
}
\end{equation}
and put $\theta(x_i) = x_i(u_1, \ldots, u_n)$ for each $1 \leq i \leq
n$. For $1 \leq i \leq n$, we apply the chain rule to find
\begin{equation}\label{Eq:chainrule}
    \dfrac{\del \theta(f)}{\del u_i} = 
    \sum_{j = 1}^n \theta \left( \dfrac{\del f}{\del x_j} \right)
    \dfrac{\del x_j}{\del u_i}.
\end{equation}
Since $\theta$ is an isomorphism, the Jacobian matrix of $\theta$
is invertible.  It follows that the image of the gradient ideal of $f$
under $\theta$ is equal to the gradient ideal of $\theta(f)$ as
required.
\end{proof}

\begin{definition}\label{D:gradscheme} 
Let $X$ be a scheme of finite
type over $k$, and let $P \in X$ be a $k$-rational point.  We say that
the pointed scheme $(X, P)$ is a \emph{gradient scheme} if the
completion of the local ring at $P$ with respect to its maximal ideal is
isomorphic as a complete local $k$-algebra to 
\[ 
    \dfrac{k[[x_1, \ldots, x_n]]}{I} 
\] 
for some gradient ideal $I \subseteq k[[x_1, \ldots, x_n]]$.  If $I$ is
a principal ideal, we say that the gradient scheme is \emph{principal}.
\end{definition}

We state a simple lemma regarding ideals in power series rings.

\begin{lemma}\label{L:princgradideal}
If an ideal $(f_1, \ldots, f_k) \subseteq k[[x_1, \ldots, x_n]]$ is
principal, then 
\[
    (f_1, \ldots, f_k) = (f_i)
\]
for some $1 \leq i \leq n$.
\end{lemma}

For any ring $R$, let $\nil R$ denote the nilradical of $R$.

\begin{lemma}\label{L:nilcomplet}
Let $(A,\mf{m})$ be a local $k$-algebra that is the localization of a
finitely generated $k$-algebra, and let $\widehat{A}$ denote the
$\mf{m}$-adic completion of $A$.  Then 
\[
    \left( \dfrac{A}{\nil A} \right)^{\widehat{\ }} \cong
    \dfrac{\widehat{A}}{\nil
    \widehat{A}}.
\]
\end{lemma}

\begin{proof}
Since completion is flat (\cite[Theorem~8.8]{hM89}) and $\widehat{I} =
I\cdot \widehat{A}$ for any ideal $I \subseteq A$
(\cite[Theorem~8.11]{hM89}), we know that 
\[
     \left( \dfrac{A}{\nil A} \right)^{\widehat{\ }} \cong
     \dfrac{\widehat{A}}{(\nil A)\cdot \widehat{A}}.
\]
Since 
\[
    \left( \dfrac{A}{\nil A} \right)
\] 
is the localization of a finitely generated $k$-algebra and an integral
domain, its completion has no nilpotent elements (\cite[Ch.~VIII,
\S~13, Theorem~32]{oZ60}). Hence, $(\nil A) \cdot \widehat{A} = \nil
\widehat{A}$ as required.
\end{proof}

\begin{theorem}\label{T:princgrad}
If a gradient scheme $(X, P)$ is principal, then its reduced
subscheme is regular at $P$.  
\end{theorem}

\begin{proof}
Let $\Oh_P$ denote the local ring at $P$.  We must show that the local
ring 
\[
    \dfrac{\Oh_P}{\nil \Oh_P}
\]
is regular.  A local ring is regular if and only if its completion is
regular (\cite[Ch.~VIII, \S~11]{oZ60}), so by Lemma~\ref{L:nilcomplet}
it suffices to show that the ring
\[
    \dfrac{\widehat{\Oh_P}}{\nil \widehat{\Oh_P}}
\]
is regular.  We are assuming that the pointed scheme $(X,P)$ is a
principal gradient scheme so we may assume, by
Lemma~\ref{L:princgradideal}, that 
\[
    \widehat{\Oh_P} \cong \dfrac{k[[x_1, \ldots, x_n]]}{(f_{x_1})}
\]
for some element $f \in k[[x_1, \ldots, x_n]]$ that is not a unit and
such that $f_{x_1} \neq 0$.  To establish the theorem we must show that
the ring
\[
    \dfrac{k[[x_1, \ldots, x_n]]}{\sqrt{(f_{x_1})}}
\]
is regular.

We proceed by analyzing the form of the element $f$.  First note that if
$f = g^k$ for some element $g \in k[[x_1, \ldots, x_n]]$ having non-zero
linear term, then the result follows from Lemma~\ref{L:gradinvar}.
Indeed, in this case there is a formal change of coordinates $(x_1,
\ldots, x_n) \mapsto (u_1, \ldots, u_n)$ under which $g^k$ is sent to
${u_1}^k$.  

We next consider the case $f \in I_{\nabla f}$.  In this case, we have
$(f, I_{\nabla f}) = (f_{x_1})$ and hence all associated primes of the ring
\[
    \dfrac{k[[x_1, \ldots, x_n]]}{(f, I_{\nabla f})} = \dfrac{k[[x_1,
    \ldots, x_n]]}{(f_{x_1})}
\]
have height one.  It follows from Theorem~\ref{T:reduced} that the ring 
\[
    \dfrac{k[[x_1, \ldots, x_n]]}{\sqrt{(f)}}
\]
is regular.  Therefore $f = g^k$ for some power series $g$ with nonzero
linear term, and we are in the case of the previous paragraph.

To complete the proof it suffices to show that we must have $f \in
(f_{x_1})$.  By Lemma~\ref{L:radgrad}, we know that $f^k \in (f_{x_1})$
for some integer $k \geq 1$.  This implies that if an irreducible power
series $h$ divides $f_{x_1}$ it must also divide $f$.  The result now
follows from Lemma~\ref{L:divpartials}. 
\end{proof}

\bibliographystyle{amsplain}
\bibliography{bib}
\end{document}